\documentclass[a4paper]{article}
\usepackage[dvipsnames]{xcolor}
\usepackage{amsmath}
\usepackage{amssymb}
\usepackage{amsthm}
\usepackage{mathrsfs}
\usepackage{enumerate}
\usepackage{color}
\usepackage{geometry}
\usepackage{bbm}
\usepackage{slashed}
\usepackage{hyperref} 
\usepackage{enumerate,cancel,amssymb,centernot}
 \usepackage{latexsym}
 \usepackage{verbatim}
\usepackage{graphicx}
\usepackage{tikz}
\RequirePackage[numbers]{natbib}
\DeclareSymbolFont{sfoperators}{OT1}{ptm}{m}{n}
\DeclareSymbolFontAlphabet{\mathsf}{sfoperators}

\usepackage{float} 
\usepackage{amsthm}
\usepackage{graphicx} 
\usepackage{mathtools, stmaryrd}
\usepackage{ulem}
\usepackage[english]{babel}
\usepackage{todonotes}
\usepackage[T1]{fontenc}
\usepackage{enumitem}

\numberwithin{equation}{section}

\newcommand{\keywords}[1]{\par\noindent\textbf{Keywords:} #1}
\newcommand{\msc}[1]{\par\noindent\textbf{MSC Classification:} #1}

\newtheorem{thm}{Theorem}[section]

\newtheorem{lem}[thm]{Lemma}
\newtheorem{pro}[thm]{Proposition}

\newtheorem{rem}[thm]{Remark}

\newtheorem*{thmA*}{Theorem A}
\newtheorem*{thmB*}{Theorem B}

\def\p{{\mathbb P}}
\def\e{{\mathbb E}}

\def\h{{\mathsf H}}
\def\z{{\mathbb Z}}
\def\r{{\mathbb R}}
\def\N{{\mathbb N}}

\makeatletter
\def\operator@font{\mathgroup\symsfoperators}
\makeatother

\title{Exact Limsup Growth of Rarely Visited Sites for One-Dimensional Simple Random Walk}
\author{Chenxu Feng  \\ Peking University \and Chenxu Hao\thanks{Corresponding author. E-mail: \texttt{haochenxu@pku.edu.cn}}  \\ Peking University}

\begin{document}

\maketitle

\begin{abstract}
We investigate the minimal local time $f(n)$ of a one-dimensional simple random walk up to time $n$, defined as the smallest number of visits to any site in the range. A conjecture formulated repeatedly by Erd\H{o}s and R\'{e}v\'{e}sz (1987, 1991) stated that $\limsup_{n\to\infty}f(n)=2$ almost surely, which was disproved by T\'{o}th (1996) who showed $\limsup_{n\to\infty}f(n)=\infty$. Subsequently, R\'{e}v\'{e}sz (2013) suggested studying the growth rate and established an upper bound of the order $\log n$.

In this paper, we determine the precise asymptotic growth rate, proving that with probability one,
$$
\limsup_{n\to\infty}\frac{f(n)}{\log\log n}=\frac{1}{\log 2}.
$$
This result answers the open question posed in Section 13.2 of R\'{e}v\'{e}sz (2013).
\end{abstract}


\msc{60F15, 60J55}

\section{Introduction}
For a discrete-time simple random walk (SRW) $(S_n)_{n\geq 0}$ on $\mathbb{Z}$ with initial state $S_0=0$, we define its local time and range as follows. For any $n\in\N$ and $x\in\z$, let 
\begin{eqnarray*}
\xi(x,n)=\#\{0\le j\le n:S_j=x\},\qquad \mathcal{R}(n)=\{x\in\mathbb{Z}:\exists i\in[0,n],S_i=x\},
\end{eqnarray*}
where $\xi(x,n)$ represents the local time (number of visits) at site $x$ up to time $n$, and $\mathcal{R}(n)$ denotes the range (set of visited sites) of the random walk by time $n$.
The minimal local time $f(n)$ is defined as:
\begin{eqnarray*}
f(n)=\max\{r:\xi(x,n)\ge r\mbox{~for~any~}x\in\mathcal{R}(n)\}
=\min\{\xi(x,n):x\in\mathcal{R}(n)\}
\end{eqnarray*}
which gives the smallest number of visits among all sites in the range at time $n$.

In 1980's, Erd\H{o}s and R\'{e}v\'{e}sz \cite{ER87,ER91} conjectured 
\begin{eqnarray*}
\limsup_{n\to\infty} f(n)=2~\quad\quad a.s.
\end{eqnarray*}
i.e., they conjectured that for any integer $r\ge3$, the probability that there are infinitely many $n$ such that each of the sites of $\mathcal{R}(n)$ has been visited at least $r$ times up to step $n$ is $0$.

T\'{o}th \cite{To96} disproved this conjecture. Actually, T\'{o}th showed for any positive integer $r$, this probability equals to $1$, i.e.,
\begin{eqnarray*}
\limsup_{n\to\infty} f(n)=\infty~\quad\quad a.s.
\end{eqnarray*}
We emphasize that this proof involves the establishment of uniform bound of the extinction probability for the critical {\it Galton-Watson} process descending from an arbitrary level $r$ to zero.
This being the case, the natural subsequent question concerns the asymptotic behavior analyzed by R\'{e}v\'{e}sz \cite[Section~13.2]{Re13}.
R\'{e}v\'{e}sz \cite[Theorem~13.12]{Re13} established almost surely that
$$\limsup_{n\to\infty} f(n)\le a\log n,$$
where $a>\frac{1}{2}$, by analyzing how many times a random walk reaches its maximum.
It is worth noting that this result appears in both the 2005 and 2013 editions of R\'{e}v\'{e}sz's book.
In this work, we will give a complete answer to this question:
\begin{thm}\label{main result}
With probability $1$,
\begin{eqnarray*}
\limsup_{n\to\infty} \dfrac{f(n)}{\log\log n}=\frac{1}{\log 2}.
\end{eqnarray*}
\end{thm}

We now outline the key steps of the proof.
Our approach begins by defining two carefully constructed sequences of stopping times that partition the random walk trajectory into well-defined segments. 
Each segment represents a path from a point sufficiently far in the positive direction to one sufficiently far in the negative direction. Since the random walk can only contribute to one side of the range at each segment (positive or negative) the introduction of these stopping times allows us to advance the work of \cite{Re13} in the upper-bound analysis.

For the lower bound, the argument requires more delicate handling. We further divide the range into three segments at two stopping times. The middle segment, which contains the origin and constitutes the majority of the range, is relatively straightforward to estimate. The remaining two segments correspond to the boundary regions of the range interval. To handle these, we employ the characterization of edge local times via the Ray-Knight theorem \cite[Theorem 1.1]{Kn63}, along with some explicit computations, to complete the proof of this part and, consequently, the lower bound.

Consider replacing $S_n$ by $S_n^{(d)}$ with $d\ge1$, a simple random walk on $\mathbb{Z}^d$. The corresponding $f^{(d)}(n)$ raises the natural question of determining $\limsup\limits_{n\to\infty} f^{(d)}(n)$.
The result relies on the following observations. 
\begin{thmA*}[\cite{ET60,Fl76}]\label{two and more dimensions}
For $k,n\in\N$, let $Q_k^{(d)}(n)$ denote the number of sites visited exactly $k$ times up to step $n$ in a $d$-dimensional simple random walk $(S_n^{(d)})$ on $\mathbb{Z}^d$. Formally,
$$Q_k^{(d)}(n)=\#\{j:0\le j\le n,\,\xi(S_j^{(d)},n)=k\}.$$
Then for $k\in\N$, we have the following asymptotic results:
\begin{eqnarray*}
&&\lim_{n\to\infty}\dfrac{Q_k^{(2)}(n)\cdot(\log n)^2}{\pi^2n}=1~\quad\quad\quad ~~~ a.s.\\
&&\lim_{n\to\infty}\dfrac{Q_k^{(d)}(n)}{n}=\gamma_d^2(1-\gamma_d)^{k-1}~\quad\quad a.s.~~~~~\mbox{ if }~~d\ge3,
\end{eqnarray*}
where $\gamma_d:= \p\left(S_n^{(d)} \neq S_0^{(d)},\; \forall n\geq 1\right)$, which is the probability that a simple random walk never returns to the starting point. 
\end{thmA*}
Theorem A ensures that, with probability 1, the range of a simple random walk in dimensions two and above contains infinitely many sites visited exactly once.
Thus, for $d\ge2$,
\begin{eqnarray*}
\limsup_{n\to\infty} f^{(d)}(n)=\lim_{n\to\infty} f^{(d)}(n)=1~~~~~~a.s.
\end{eqnarray*}
We see that the limit properties of $Q_{k}^{(2)}(n)$ do not depend on $k$, as explained in \cite{Ha97}. Further analysis of $Q_{k}^{(d)}(n)$ can be found in \cite{Pi74,Ha92}.
We now focus on the one-dimensional SRW. Let $g_1(n)$ denote the number of sites visited exactly once up to time $n$, defined as: $g_1(n):=\#\{x:\xi(x,n)=1\}$. \cite{Ma88,Ne84} discovered the following several interesting results on $g_1(n)$.
\begin{thmB*}[\cite{Ma88,Ne84}]\label{once visited site}
There exists some constant $0<C<\infty$ such that
\begin{eqnarray*}
\limsup_{n\to\infty} \dfrac{g_1(n)}{(\log n)^2}=C~~~~~~a.s.
\end{eqnarray*}
And for $n\in\N$,
\begin{eqnarray*}
\e\, g_1(n)=2.
\end{eqnarray*}
\end{thmB*} 
If we consider the opposite of ``rarely visited'' sites, we arrive at the concept of favorite sites (also called most visited sites). Formally, we define the set of favorite sites at time 
$n$ as:
$$\mathcal{K}^{(d)}(n):=\{x\in\mathbb{Z}^d:\xi(x,n)=\sup_{y\in{\mathbb{Z}^{d}}}\xi(y,n)\}.$$
The study of $\mathcal{K}^{(d)}(n)$ was initiated by
Erd\H{o}s and R\'{e}v\'{e}sz \cite{ER84,ER87}, where they formulated several fundamental questions. 
Despite significant progress over the years (see \cite{ER87,To01,DS18,HLOZ24} for ``cardinality'' and \cite{BG85,LS04,Ba23,DPRZ01} for ``escape rate''), many of these original questions remain open to this day.  For detailed surveys, see \cite{ST00,Oka16}.

We now briefly discuss the organization of this work. Section \ref{se:2} is dedicated to preliminaries. The upper and lower bounds of Theorem \ref{main result} will be proved in Sections \ref{se:3} and \ref{se:4} respectively.

\section{Preliminaries}\label{se:2}

We denote by $c$ and $C$ positive and finite constants whose values are universal and may change from line to line.
If $\{a_n\}$ and $\{b_n\}$ are non-negative sequences, then we write
$a_n\lesssim b_n$ if there exists $c>0$ such that $a_n \leq c\, b_n$ for all $n$, and $a_n \asymp b_n$ if $a_n\lesssim b_n$ and $b_n\lesssim a_n$. For a sequence $\{a'_n\}$, we write $a'_n=O(b_n)$ if $|a'_n| \lesssim b_n$.
For a sequence $\{c_n\}$, we write $c_n=o(b_n)$ if $\lim_{n\to\infty}\frac{b_n}{c_n}=\infty$.

We now turn to random walk. 
Recall that $(S_n)_{n\geq 0}$ is a SRW on $\mathbb{Z}$, starting at $S_0=a$ and let $\p^a$ stand for the respective probability measure. We write $\p=\p^0$ for short. 
For  $x\in\mathbb{Z}$, $n\in\mathbb{N}$, we use $\h$ to denote the first hitting time. More precisely,
\begin{eqnarray*}
\h_{x}:=\inf\{k>0: S_k=x\};~\h_{x,y}:=\inf\left\{k>0: S_k\in\{x,y\}\right\};~\h_x(n):=\inf\{k\geq n: S_k=x\}-n
\end{eqnarray*}
for the first hitting time (after time $n$).
We also define corresponding $\sigma$-algebras:
\begin{equation}\label{sigma}
\mathcal{F}_{T}:=\sigma\{S_{[0,T]}\},
\end{equation}
where $S_{[0,T]}$ denotes the sub-path $\{S_{0},S_{1},...,S_{T-1},S_{T}\}$.
Below we introduce the following forms of the Law of Iterated Logarithm: the Khinchine (See e.g. \cite[Section 4.4]{Re13}) and the Chung (See e.g. \cite[Section 5.3]{Re13}). Let $L_n=\max_{0\le k\le n}|S_k|$, then almost surely:
\begin{eqnarray}\label{classic LIL}
\limsup_{n\to\infty}\frac{L_n}{\sqrt{2n\log\log n}}=1,~~~\liminf_{n\to\infty}\frac{L_n}{\sqrt{\frac{n}{\log\log n}}}=\frac{\pi}{\sqrt{8}}.
\end{eqnarray}
Therefore, \eqref{classic LIL} implies that, with probability 1,
\begin{eqnarray}\label{range LIL}
\limsup_{n\to\infty}\frac{\mathcal{R}(n)}{\sqrt{2n\log\log n}}\le2,~~~\liminf_{n\to\infty}\frac{\mathcal{R}(n)}{\sqrt{\frac{n}{\log\log n}}}\ge\frac{\pi}{\sqrt{8}}.
 \end{eqnarray}

\section{Upper Bound}\label{se:3}
Denote $T_0=N_0=0$. For each $j\in\mathbb{N}$, we introduce the following notation (See Figure~\ref{fig:stopping_time} for illustration):
\begin{eqnarray*}
&&T_j':=\inf\{n>T_j:S_n\le N_j\},~M_j:=\sup_{0\le n\le T_j'}S_n;\\
&&T_{j+1}:=\inf\{n>T_j':S_n\ge M_j\},~N_{j+1}:=\inf_{0\le n\le T_{j+1}}S_n.
\end{eqnarray*}

\begin{figure}[htbp]
    \centering
    \includegraphics[scale=0.45]{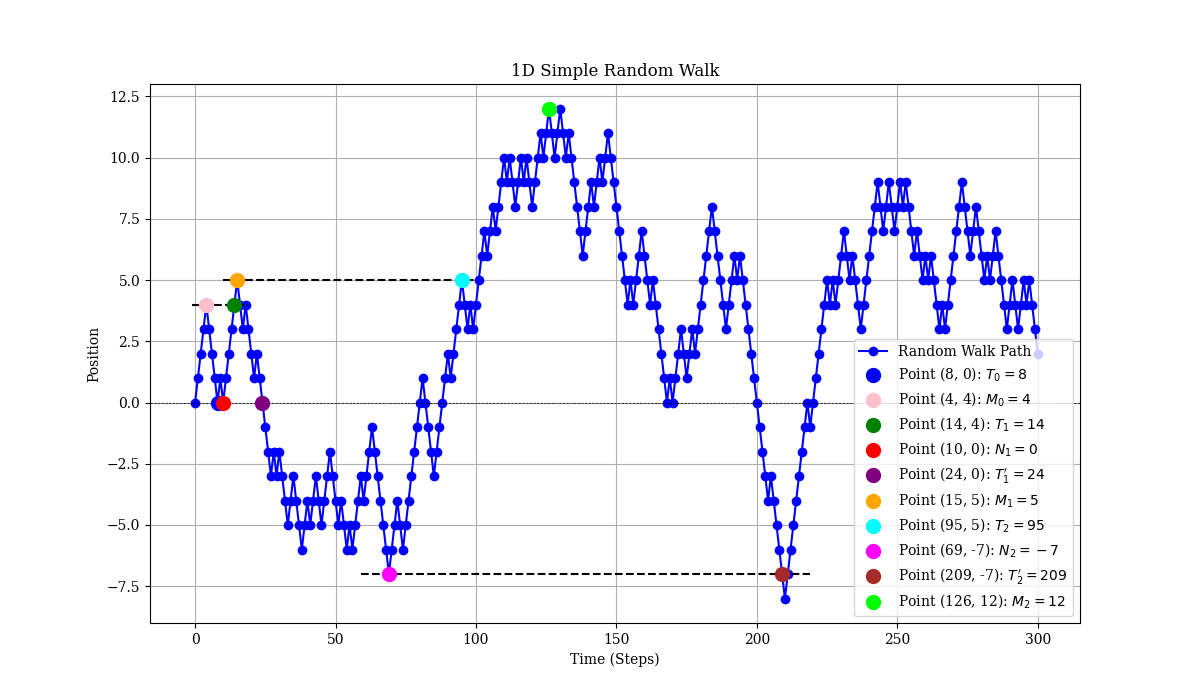} 
    \caption{In this visualization, $T_3>300,N_3\le -8$.}
    \label{fig:stopping_time} 
\end{figure}
It is easy to verify that for $j\in\mathbb{N}$:
\begin{eqnarray*}
M_j=\sup_{T_j\le n\le T_j'}S_n,~~\mbox{and}~~N_{j+1}=\inf_{T_j'\le n\le T_{j+1}}S_n.
\end{eqnarray*}
Since the random walk from time $T_{j}'$ to time $T_{j+1}$ describes a path starting at $N_j$ and ending at $M_{j}$, the following holds:
\begin{eqnarray}\label{Upper-N}
    \p\left(N_{j+1}\leq N_j-l\Big|\mathcal{F}_{T_j'}\right)=\p\left(\h_{-l}<\h_{M_j-N_j}\right)=\frac{M_j-N_j}{M_j-N_j+l}.
\end{eqnarray}
Similarly, 
\begin{eqnarray}\label{Upper-M}
    \p\left(M_{j+1}\geq M_j+l\big|\mathcal{F}_{T_{j+1}}\right)=\frac{M_j-N_{j+1}}{M_j-N_{j+1}+l}.
\end{eqnarray}
We introduce $K_j$ and $K_j'$ to represent the logarithm of the distance from  $S_{T_j'}$ to $S_{T_{j+1}}$ and the logarithm of the distance from $S_{T_{j+1}}$ to $S_{T_{j+1}'}$, respectively. More precisely,
\begin{eqnarray}\label{def-Kn}
   K_j=\log(M_j-N_j)~~\mbox{and}~~ K_j'=\log(M_j-N_{j+1}).
\end{eqnarray}
For $l\in\N$, by \eqref{Upper-N} and \eqref{Upper-M},
\begin{eqnarray*}
&&\p\left((K_j'-K_j)\geq \log\left(\frac{e^{K_j}+l}{e^{K_j}}\right)\Big|\mathcal{F}_{T_j'}\right)=\frac{e^{K_j}}{e^{K_j}+l},\\
&&\p\left((K_{j+1}-K_j')\geq \log\left(\frac{e^{K_j'}+l}{e^{K_j'}}\right)\Big|\mathcal{F}_{T_{j+1}}\right)=\frac{e^{K_j'}}{e^{K_j'}+l}.
\end{eqnarray*}
Then we have,
\begin{eqnarray}\label{SLLN1}
\e(K_j'-K_j|\mathcal{F}_{T_j'})=1-O(e^{-K_j}),~~ \e(K_{j+1}-K_j'|\mathcal{F}_{T_{j+1}})=1-O(e^{-K_j'}).
\end{eqnarray}

Combining \eqref{SLLN1} and the Strong law of Large Numbers:
\begin{eqnarray}\label{K_n LLN}
&&\lim_{j\to\infty}\dfrac{K_j}{j}=2~~~a.s.
\end{eqnarray}
Using \eqref{def-Kn}, we know that
\begin{equation*}
\lim_{j\to\infty}\frac{\log M_j}{2j}=\lim_{j\to\infty}\frac{\log N_j}{2j}=1 \quad \text{a.s.}
\end{equation*}
Hence, \eqref{K_n LLN} and  \eqref{range LIL} implies that
\begin{eqnarray}\label{SLLN-Tn}
&&\lim_{j\to\infty}\dfrac{\log T_j}{j}=4~~~a.s.
\end{eqnarray}
 
Since 
\begin{eqnarray}\label{N_k}
    \p(N_k<N_{k-1}|\mathcal{F}_{T_{k-1}'})=1-\frac{1}{M_{k-1}-N_{k-1}+1},
\end{eqnarray}
we see that $N_k=N_{k-1}$ holds for only finitely many $k$ with probability 1.
Let us define the stopping time $\tau_r=\inf\{n>0:f(n)=r\}$. We can assume that there exists a positive integer $k$ such that $T_{k-1}\leq \tau_r \le T_{k}'-1$. In particular, when $N_{k}<N_{k-1}$, $N_{k}$ is visited at least $r$ times within $[T_{k},T_{k}']$. 
To conclude this section, we restrict our attention to this case. 
(If $\tau_r\in[T_{k-1}',T_k]$, we consider $M_k$ is visited at least $r$ times.) 
Under this assumption, we consider $\tau_{r} \in \left[T_{k}, T_{k}'\right]$. More precisely, there exists $N_1 = N_1(\omega) \in \mathbb{N}$ such that for all $r > N_1$,
$$\tau_{r}\ge\sup\{T_k:\inf_{0\le n\le T_{k-1}}S_n=\inf_{0\le n\le T_{k}}S_n\}\quad\mbox{and}\quad\tau_{r}\in \left[T_{k},T_{k}'\right].$$

Note that if we already know the value of $N_k$, then the distribution of the random walk after the first time hitting $N_k$ is a SRW condition on hitting $M_k$ before hitting $N_{k}-1$, which implies that 
\begin{align*}
     &\p(N_k \text{ has been visited at least $r$ times before $T_k'-1$},N_{k}<N_{k-1})\\
     \leq&\,\,\p(N_k \text{ has been visited at least $r$ times in $[T_{k-1}',T_k]$})\\
     \leq&\,\,\p\left((X_{n})\text{ visited $0$ at least r times}\right)=\frac{1}{2^{r-1}},
\end{align*}
where $(X_n)$ is a SRW on $\mathbb{Z}$ starting at $0$, condition on never hitting $-1$.
It follows that
\begin{align*}
    \p(T_{k-1}'\leq\tau_r \leq T_{k}'-1, N_k<N_{k-1})\le\frac{1}{2^{r-1}}.
\end{align*}
And we know $N_k<N_{k-1}$ for large $k$. Therefore, for large $r$,
\begin{eqnarray*}
    \p(\tau_r< T_{[\frac{2^r}{r^2}]})\le\sum_{k=1}^{[\frac{2^r}{r^2}]}\p(T_{k-1}'\leq\tau_r \leq T_{k}'-1)\leq \frac{2^r}{r^2}\cdot\frac{1}{2^{r-1}}=\frac{2}{r^2}.
\end{eqnarray*}
Combining this with the Borel-Cantelli lemma implies that there almost surely exists an integer $N$ such that $\tau_r\geq T_{[\frac{2^r}{r^2}]}$ for all $r\geq N$. Finally, with probability 1,
\begin{eqnarray}\label{upper-bound}
    \limsup_{n\to\infty} \dfrac{f(n)}{\log\log n}=\limsup_{r\to\infty} \dfrac{f(\tau_r)}{\log\log \tau_r}\leq \limsup_{r\to\infty} \dfrac{r}{\log\log T_{[\frac{2^r}{r^2}]}}\overset{(\ref{SLLN-Tn})}{=}\frac{1}{\log 2}.
\end{eqnarray}
This finishes the proof of the upper bound.

\section{Lower Bound}\label{se:4}
Since we are considering the limit in the $\limsup$ sense (which concerns large values of $r$), we can assume without loss of generality that there exists a sufficiently large constant $K_1$ such that $r > K_1$.
The main contribution to the lower bound comes from the SRW visiting the maximum or the minimum in the range, i.e., $M_j$ or $N_{j+1}$, at least $r$ times. We note that the probability of the SRW visiting $N_{j+1}$ between $T_j$ and $T_{j+1}$ is of order $2^{-r}$ (see \eqref{smallest distance estimate}). 
Moreover, conditional on this event, the probability that it also visits other sites in the range at least $r$ times is substantially large, as shown in Proposition \ref{main-pro} and Lemmas \ref{lem01} and \ref{lem02}.
This observation naturally gives rise to the following event:
\begin{eqnarray*}
&&A_{j,r}=\Big\{\forall x\in[N_{j+1}+1,M_j-r^4]:\#\{s:T_j\le s\le T_{j+1},S_s=x\}\ge r\Big\};\\
&&B_{j,r}=\Big\{\#\{s:T_j\le s\le T_{j+1},S_s=N_{j+1}\}\ge r\Big\}.
\end{eqnarray*}

Let $\mathcal{F}_{M_j}^s$ denote the $\sigma$-algebra generated by the value of $M_j$ before $T_j'$, and let $\mathcal{F}_{N_{j+1}}^s$ denote the $\sigma$-algebra generated by the value of $N_{j+1}$.
Define the sigma field $\mathcal{F}'_j=\sigma(\mathcal{F}_{T_j}\cup \mathcal{F}_{M_j}^s \cup \mathcal{F}_{N_{j+1}}^s)$. Then for large $j\ge r^5$ and by \eqref{K_n LLN},
\begin{eqnarray}\label{smallest distance estimate}
\p(B_{j,r}|\mathcal{F}'_j)=\left(\frac12\cdot\p^{N_{j+1}+1}\left(\h_{N_{j+1}}<\h_{M_{j}+1}\right)\right)^{r-1}=\left(\frac{1}{2}\cdot \frac{M_{j}-N_{j+1}}{M_{j}-N_{j+1}+1}\right)^{r-1}\asymp 2^{-r}.
\end{eqnarray}
Let $\mathcal{F}_j^L=\sigma(\mathcal{F}'_j\cup \{B_{j,r}\})$.
The proof of the lower bound in Theorem \ref{main result} is postponed until the end of this section. We first turn to the next proposition, as it is a prerequisite for the proof.
\begin{pro}\label{main-pro}
On the event $B_{j,r}$ and $j\ge r^5$, we have 
$$\p(A_{j,r}|\mathcal{F}_j^L)\geq Cr^{-2}.$$
\end{pro}
We split $A_{j,r}$ into two parts:
\begin{eqnarray*}
&&A_{j,r}^{1}=\Big\{\forall x\in[N_{j+1}+r^4,M_j-r^4]:\#\{s:T_j\le s\le T_{j+1},S_s=x\}\ge r\Big\},\\
&&A_{j,r}^{2}=\Big\{\forall x\in[N_{j+1}+1,N_{j+1}+r^4-1]:\#\{s:T_j\le s\le T_{j+1},S_s=x\}\ge r\Big\}.
\end{eqnarray*}
And Proposition \ref{main-pro} can be proved by verifying the following two Lemmas.
\begin{lem}\label{lem01}
On the event $B_{j,r}$ and $j\ge r^5$, we have
$$\p(A_{j,r}^1|\mathcal{F}_j^L)\geq 1-3r^{-2}.$$
\end{lem}
\begin{lem}\label{lem02}
On the event $B_{j,r}$ and $j\ge r^5$, we have 
$$\p(A_{j,r}^2|\mathcal{F}_j^L)\geq \frac{C}{r}.$$
\end{lem}
Our proofs of Lemmas \ref{lem01} and \ref{lem02} both rely on the assumption that event $B_{j,r}$ holds.
To prepare for the proof, we begin by introducing the following notation: 
\begin{eqnarray}
U_j=\sup\{n<T_j':S_n=M_j\};~V_j^0=U_j,~V_j^i=\inf\{n>V_j^{i-1}:S_n=N_{j+1}\}\mbox{~with~}1\le i\le r.
\end{eqnarray}
It is important to note that neither $U_j$ nor $V_j^i$ with $j,i\in\N$ is a stopping time, and $V_j^r$ is not necessarily the last hitting time of $N_{j+1}$ before $T_{j+1}$. 
Within the time interval $[U_j,T_{j+1}]$, the SRW can be decomposed into three components by partitioning at the times $V_j^1$ and $V_j^r$. Specifically, we define:
\begin{align*}
&S_{n}^1=S_{n+U_j},\quad\mbox{~~for~~} 0\leq n\leq V_j^1-U_j,\\
&S_{n}^2=S_{n+V_j^1},\quad\mbox{~~for~~} 0\leq n\leq V_j^r-V_j^1,\\
&S_{n}^3=S_{n+V_j^r},\quad\mbox{~~for~~} 0\leq n\leq T_{j+1}-V_j^r.   
\end{align*}
These three components are independent. For ease of subsequent proofs, we perform a time reversal of $(S_n^1)$ and denote the reversed process by $(R_n)$, i.e.,
\begin{eqnarray}
    R_n=S_{V_j^1-U_j-n}^1~~~\text{for}~U_j\le n\le V_j^1.
\end{eqnarray}
\begin{rem}
On the event $B_{j,r}$, we see that
\begin{itemize}
\item $(R_n)$ is a SRW starting at $N_{j+1}$, conditioned to reach $M_{j}$ before returning to $N_{j+1}$, and stopped when hitting $M_{j}$.
\item $(S_n^2)$ is constructed by combining $r-1$ independent SRWs, each starting at $N_{j+1}$, which first step to $N_{j+1}+1$ and then behave as a SRW conditioned to return to $N_{j+1}$ before reaching $M_j$ stopping upon re-entry at $N_{j+1}$.
\item $(S_n^3)$ is a SRW starting at $N_{j+1}$ and conditioned to reach $M_{j}$ before hitting $N_{j+1}-1$.
\end{itemize}  
\end{rem}

\begin{proof}[Proof of Lemma~\ref{lem01}]
For any $x\in [N_{j+1}+r^4,M_j-r^4]$, denote 
\begin{eqnarray*}
\xi_1(x)=\#\{k:0\le k\le V_j^1-U_j,R_k=x\},~\xi_2(x)=\#\{k:0\le k\le T_{j+1}-V_j^r,S_k^3=x\}.
\end{eqnarray*}
Then on the event $B_{j,r}$,
\begin{eqnarray}\label{proof of lower bound 01}
\p (\xi_1(x)\geq r)&=&\p^x\Big(\mbox{Before reaching } \{M_j,N_{j+1}\}, ~\mbox{SRW first hits }x \mbox{ at least } r-1 \mbox{ times}\Big)\nonumber\\
&=&\Big(1-\frac12\cdot \p^{x+1}\big(\mbox{Before returning } x ,~\mbox{SRW first hits }M_j \big)\nonumber\\
&&~~~-\frac12\cdot\p^{x-1}\big(\mbox{Before returning } x ,~\mbox{SRW first hits }N_{j+1} \big)\Big)^{r-1}\nonumber\\
&=&\left( 1-\frac12\cdot\frac{1}{M_j-x}-\frac12\cdot\frac{1}{x-N_{j+1}}\right)^{r-1}\nonumber\\
&=&\left( 1-\frac{M_j-N_{j+1}}{2(M_j-x)(x-N_{j+1})}\right)^{r-1}.
\end{eqnarray}
Similarly,
\begin{eqnarray}\label{proof of lower bound 02}
\p (\xi_2(x)\geq r)=\left( 1-\frac{M_j-N_{j+1}+1}{2(M_j-x)(x-N_{j+1}+1)}\right)^{r-1}.
\end{eqnarray}
For $x\in [N_{j+1}+r^4,M_j-r^4]$, it is clear that 
\begin{eqnarray}\label{proof of lower bound 03}
\#\{k:T_j\le k\le T_{j+1},S_k=x\}\geq \xi_1(x)+\xi_2(x).
\end{eqnarray}
Then for large $j\ge r^5$,
\begin{equation*}
\begin{split}
& \p\left(\#\{k:T_j\le k\le T_{j+1},S_k=x\}\geq r\right)\\
\overset{(\ref{proof of lower bound 03})}{\ge}\;\,& 1-\p\left(\xi_1(x)+\xi_2(x)<r\right)\\
\ge \;~~& 1-\left(1-\p (\xi_1(x)\geq r))\cdot(1-\p (\xi_2(x)\geq r)\right)\\
\overset{(\ref{proof of lower bound 01})}{\underset{(\ref{proof of lower bound 02})}{\ge}}\;\,&1-\,r^2\frac{M_j-N_{j+1}}{2(M_j-x)(x-N_{j+1})}\cdot\frac{M_j-N_{j+1}+1}{2(M_j-x)(x-N_{j+1}+1)}\\
\ge\;~~&1-\frac{r^2}{\min\{M_j-x,x-N_{j+1}\}(\min\{M_j-x,x-N_{j+1}+1\})},
\end{split}
\end{equation*}
where the last inequality is derived from the inequality $(a+b)\min\{a,b\}\le 2ab$, which holds for all $a,b\in\mathbb{N}$.
By applying the union bound, 
\begin{eqnarray*}
\p(A_{j,r}^1|\mathcal{F}_j^L)&\ge &1-\sum_{x=N_{j+1}+r^4}^{M_j-r^4}\left(\frac{r^2}{\min\{M_j-x,x-N_{j+1}\}(\min\{M_j-x,x-N_{j+1}+1\})}\right)\\
&\ge& 1-\sum_{k=r^4}^{\lceil\frac{M-j-N_{j+1}}{2}\rceil} \frac{r^2}{k^2}-\sum_{k=r^4}^{\lceil\frac{M-j-N_{j+1}}{2}\rceil} \frac{r^2}{k(k+1)}\ge 1-3r^{-2},
\end{eqnarray*}
where $\lceil a\rceil=\min\{ n \in \mathbb{Z}: n \geq a \}$ with $a\in\r$.
This concludes the proof.
\end{proof}

\begin{proof}[Proof of Lemma~\ref{lem02}]
We begin by defining three corresponding random walks associated with $\{(R_n),(S_n^2),(S_n^3)\}$. Firstly, define $(\tilde{R}_n)$ as a one-dimensional SRW starting at $N_{j+1}$ condition on never hit $(-\infty,N_{j+1}]$ again. For $i\ge1$, the transition probabilities are:
    \begin{align*}\label{upward shift transition probability}
        \p\left(\tilde{R}_{n+1}=N_{j+1}+i-1\Big|\tilde{R}_n=N_{j+1}+i\right)=\frac{i-1}{2i};\nonumber\\
        \p\left(\tilde{R}_{n+1}=N_{j+1}+i+1\Big|\tilde{R}_n=N_{j+1}+i\right)=\frac{i+1}{2i},
    \end{align*}
with initial step (forced upward transition):
$$ \p\left(\tilde{R}_{1}=N_{j+1}+1\right)=1.$$

Secondly, for $1\le i\le r-1$, let $(\tilde{U}_n^i)_{n\ge1}$ be independent random walks initialized at $\tilde{U}_1^i=N_{j+1}, \tilde{U}_2^i=N_{j+1}+1$ and evolving as SRW stopped at $N_{j+1}.$
Define $(\tilde{S}_n^2)$ is combine of $\{(\tilde{U}_n^i)_{n\ge0}\}$ with $1\leq i \leq r-1$. In other words, define the concatenated process $(\tilde{S}_n^2)$ recursively as follows:
\begin{align*}
\tilde{S}_0^2 = N_{j+1},\
\tilde{S}_n^2 = \tilde{U}_{n - \sum_{m=1}^{i-1} H_{N_{j+1}}^{m}}^{i} \quad \text{for} \quad \sum_{m=1}^{i-1} H_{N_{j+1}}^{m} < n \leq \sum_{m=1}^i H_{N_{j+1}}^{m}, \quad i = 1, \dots, r-1,
\end{align*}
where $H_{N_{j+1}}^{i}$ is the stopping time of the $i$-th process (i.e., $\tilde{U}_{H_{N_{j+1}}^{i}}^i=N_{j+1}$).

Finally, define $(\tilde{S_n^3})$ is a SRW starting at $N_{j+1}$ condition on never hitting $N_{j+1}-1$ again, i.e., for $i\ge1$:
\begin{align*}
        &\p\left(\tilde{S_n^3}=N_{j+1}+i-2\Big|\tilde{S_n^3}=N_{j+1}+i-1\right)=\frac{i-1}{2i},\\
&\p\left(\tilde{S_n^3}=N_{j+1}+i\Big|\tilde{S_n^3}=N_{j+1}+i-1\right)=\frac{i+1}{2i}.
    \end{align*}
    
    Define $H_{M_j}^R=\inf\{n>0:\tilde{R}_n=M_j\}$ as the first time when $\tilde{R_n}$ hit $M_j$, and $H_{M_j}^{S^3}=\inf\{n>0:\tilde{S_n^3}=M_j\}$ as the first time when $\tilde{S_n^3}$ hit $M_j$.
\begin{rem}\label{same distribution}
For these three random walks, we see that
\begin{itemize}
    \item $(R_n)$ has the same distribution as $(\tilde{R}_n,\ 0\leq n\leq H_{M_j}^R)$.
    \item $(S_n^2)$ has the same distribution as $(\tilde{S}_n^2)$ conditioned on not hitting $M_j$.
    \item $(S_n^3)$ has the same distribution as $(\tilde{R}_n^3,\ 0\leq n\leq H_{M_j}^{S^3})$
\end{itemize}
\end{rem}
The analysis is restricted to edges in  $[N_{j+1}, N_{j+1} + r^4 - 1]$, where the event is defined as
$$D_{j,r}=\{(\tilde{R}_n)\mbox{~or~}(\tilde{S}_n^3) \mbox{~will~hit~} N_{j+1}+r^4 \mbox{~after~hitting~}M_j\}\cup\{(\tilde{S}^2_n)\mbox{~will~hit~}M_j\}.$$
We claim that
\begin{eqnarray}\label{Does not affect hitting}
    \p(D_{j,r}|\mathcal{F}_j^L)\leq \frac{2r^4+r}{M_j-N_{j+1}}.
\end{eqnarray}
Firstly,
\begin{eqnarray*}
\p\left((\tilde{S}^2_n)\mbox{~will~hit~}M_j\right)\leq \sum_{i=1}^{r-1}\p^{N_{j+1}+1}\left(S_{H_{M_j,N_{j+1}}^{i}}=M_{j}\right)\le\frac{r}{M_j-N_{j+1}}.
\end{eqnarray*}
To calculate another part, we define a martingale $Y_n=\frac{1}{\tilde{R}_{n+H}-N_{j+1}}$ where $H=H_{M_j}^R$. Also define the stopping time $\h^Y=\inf\{k\ge0:Y_k\ge\frac{1}{r^4}\}$, then $Y_0=\frac{1}{M_j-N_{j+1}}$, and
$$\p\left((\tilde{R}_n)\mbox{~will~hit~} N_{j+1}+r^4 \mbox{~after~time~} H \right)=\p\left(H^Y<\infty\Big|Y_0=\frac{1}{M_j-N_{j+1}}\right).$$
By the optional stopping theorem,
$$\frac{1}{M_j-N_{j+1}}=\p\left(H^Y<\infty\Big|Y_0=\frac{1}{M_j-N_{j+1}}\right)\cdot\e\left(Y_{\h_Y}\Bigg|\Big\{Y_0=\frac{1}{M_j-N_{j+1}}\Big\}\cap\{\h_Y<\infty\}\right).$$
This yields the bound
$$\p\left(H^Y<\infty\Big|Y_0=\frac{1}{M_j-N_{j+1}}\right)=\frac{r^4}{M_j-N_{j+1}},$$
and similar arguments apply to $(\tilde{S}_n^3)$, which completes the proof of \eqref{Does not affect hitting}.

Remark \ref{same distribution} implies that, conditioned on $D_{j,r}^c$, the edge local times over $[N_{j+1}+1, N_{j+1}+r^4-1]$ are identically distributed in $(R_n, S_n^2, S_n^3)$ and $(\tilde{R}_n, \tilde{S}_n^2, \tilde{S}_n^3)$.
For a random walk $(\tilde{S}_n)\in \{(\tilde{R}_n),(\tilde{S}_n^2),(\tilde{S}_n^3)\}$ and $k\in\N$, define
\begin{eqnarray*}
    &&\tau_{i,\tilde{S}_n}^0=0,~\tau_{i,\tilde{S}_n}^k=\inf\{n>\tau^{k-1}_{i,\tilde{S}_n}:\tilde{S}_{n-1}=N_{j+1}+i-1,\tilde{S}_n=N_{j+1}+i\};\\    
    &&s_{i,\tilde{S}_n}=\sup\{k:\tau_{i,\tilde{S}_n}^k<\infty\}=\#\{n:\tilde{S}_{n-1}=N_{j+1}+i-1,\tilde{S}_{n}=N_{j+1}+i\}.\\
\end{eqnarray*}
Define $v_i=s_{i,\tilde{R}_n}+s_{i,\tilde{S}_n^2}+s_{i,\tilde{S}_n^3}$. 
By \cite[Theorem 1.1]{Kn63}, the local time process stopped at inverse local times allows us to rewrite these expressions in terms of the critical Galton-Watson process.
This connection enables a reformulation of the aforementioned expressions.
Thus,
\begin{eqnarray*}
    v_{i+1}=\sum\limits_{j=1}^{v_i}X_{i,j}+2,
\end{eqnarray*}
where $v_1\geq r$ (condition on $B_{j,r}$) and $X_{i,j} (i,j\in\N)$ are i.i.d. random variables with $\p(X_{i,j}=k)=2^{-k-1}$. 
Define $\tau=\inf\{k>0:v_{j}<r\}$. For large $r$, let
$$f(x) = 
\begin{cases} 
x^{-0.1} & \text{if } x \ge r-1 \\
(r-1)^{-0.1} & \text{if } x < r-1 
\end{cases}$$
i.e., $x^{-0.1}$ can be truncated to $(r-1)^{-0.1}$ for values of $x\le r-1$. Also define,
$$W_j={v_j}^{-0.1}\cdot\textbf{1}_{\{\tau>j\}}+(r-1)^{-0.1}\cdot\textbf{1}_{\{\tau\le j\}}.$$
On ${\tau > j}$, we have $W_{j+1} = f(v_{j+1})$. Notice that
$$f(v_{j+1})\le v_j^{-0.1}-0.1v_j^{-1.1}(v_{j+1}-v_j)+0.085v_j^{-2.1}(v_{j+1}-v_j)^2+(r-1)^{-0.1}\textbf{1}_{\{v_{j+1}\le0.9v_j\}}.$$
By Stirling's formula,
$$\p\left(v_{j+1}\le0.9v_j\right)\le\exp(-cv_j).$$
Therefore, $(W_j)$ is a supermartingale.
By Doob inequality,
\begin{eqnarray*}
r^{-0.1}\cdot\p\left(\max_{0\le k\le r^4}W_k\ge r^{-0.1}\right)\le\e(W_{r^4})\le Cr^{-0.4}.
\end{eqnarray*}
By $\p(\tau<r^4)=\p(\max_{0\le k\le r^4}W_k\ge r^{-0.1})$,
\begin{eqnarray*}
\p(\tau<r^4)\le Cr^{-0.3}.
\end{eqnarray*}
Then on the event $B_{j,r}$ and $j\ge r^5$, by Remark \ref{same distribution}, 
$$\p(A_{j,r}^2|\mathcal{F}_j^L)\geq  1-\left(\p(\tau\leq r^4)+\p(D_{j,r}|\mathcal{F}_j^L)\right)\geq \frac{C}{r}.$$
\end{proof}

We now establish the lower bound by applying Proposition \ref{main-pro}.
Define the stopping time $\sigma_r=\inf\{j>2^rr^{19}:A_{j,r}\}$, then by \eqref{smallest distance estimate},
\begin{align*}
\p(\sigma_r> 2^rr^{20})\le(1-c_1r^{-2}2^{-r})^{2^r(r^{20}-r^{19})}\le\exp(-c_2r^{18}).
\end{align*}
Thus,
$$\p(\sigma_r\leq 2^rr^{20})\geq 1-Cr^{-2}.$$

Define $\tilde{S}_n=-S_{n+T_{\sigma_r+1}}+S_{T_{\sigma_r+1}}$, then $(\tilde{S}_n)$ is a SRW independent with $\mathcal{F}_{T_{\sigma_r+1}}$. We apply the same stopping times $T_j,T_j'$ on $\tilde{S}_n$, note as $\tilde{T}_j,\tilde{T}_j'$, i.e., $\tilde{T}_0=\tilde{N}_0=0$ and for $j\ge1$:
\begin{eqnarray*}
&&\tilde{T}_j':=\inf\{n>\tilde{T}_j:\tilde{S}_n=\tilde{N}_j-1\},~\tilde{M}_j:=\sup_{0\le n\le \tilde{T}_j'}\tilde{S}_n;\\
&&\tilde{T}_{j+1}:=\inf\{t>\tilde{T}_j':\tilde{S}_n=\tilde{M}_j+1\},~\tilde{N}_{j+1}:=\inf_{0\le n\le \tilde{T}_{j+1}}\tilde{S}_n.
\end{eqnarray*}

Then for $(\tilde{S}_n)$, define 
$\tilde{A}_{j,r}:=\left\{\forall x\in[\tilde{N}_{j+1}+1,\tilde{M}_j-r^4]:\#\{s:\tilde{T}_j\le s\le \tilde{T}_{j+1},\tilde{S}_s=x\}\ge r\right\}$
and stopping time $\tilde{\sigma}_r=\inf\{ j>r^5:\tilde{A}_{j,r}\}$. Similarly, 
\begin{eqnarray*}
\p(\tilde{\sigma}_r\leq 2^rr^{15})\geq 1-Cr^{-2}
\end{eqnarray*}
and
\begin{eqnarray}\label{tilde LLN}
\lim_{j\to\infty}\dfrac{\log(\tilde{M}_j-\tilde{N}_{j+1})}{j}=2,~~~\lim_{j\to\infty}\dfrac{\log\tilde{T}_j}{j}=\lim_{j\to\infty}\dfrac{\log\tilde{T}_j'}{j}= 4~~~~~~a.s.
\end{eqnarray}
Thus, there exists a $K_{2}$, such that for $r>K_{2}$,
\begin{equation*}
T_{\sigma_r+1}'-T_{\sigma_r+1}>2^{2^rr^{18}},~~~\tilde{T}_{\tilde{\sigma}_r+1}<4^{2^rr^{16}}.
\end{equation*}
Thus, $T_{\sigma_r+1}+\tilde{T}_{\tilde{\sigma}_r+1}< T_{{\sigma_r}+1}'$ for all but finitely $r$.
Take $a_r=T_{\sigma_r+1}+\tilde{T}_{\tilde{\sigma}_r+1}$, we have $\liminf \limits_{r\rightarrow \infty} f(a_r)\ge r$.
By \eqref{SLLN-Tn} and \eqref{tilde LLN},
$\limsup \limits_{r\rightarrow \infty}\frac{\log\log(a_r)}{r}\leq \log2.$
Thus, 
$$\limsup_{n\to\infty}\frac{f(n)}{\log\log n}=\limsup_{r\to\infty}\frac{f(a_r)}{\log\log a_r}\ge\frac{1}{\log2}.$$
This completes the proof of the lower bound.

\section*{Acknowledgement} We thank Xinyi Li for many useful comments on an early
version of the manuscript. Hao is supported in part by China Postdoctoral Science Foundation (No.\ GZC20230089). All authors are partially supported by the National Natural Science Foundation of China (NSFC) through the NSFC Key Program (Grant No.\ 12231002).

\end{document}